\documentclass[11pt,a4paper,reqno]{amsart}
 \usepackage{amsaddr}
\usepackage{amscd,amsmath}
\usepackage{amssymb,amsmath,latexsym,color,enumitem,tikz,dsfont,tikz-cd}
\usepackage[all]{xypic}       
\usepackage[varg]{pxfonts}
\usepackage{tabu,adjustbox,booktabs}
\usepackage{graphicx}

\usepackage{hyperref}

\usetikzlibrary{decorations.pathmorphing,shapes}

\usetikzlibrary{decorations.markings}
\tikzset{negated/.style={
        decoration={markings,
            mark= at position 0.5 with {
                \node[transform shape] (tempnode) {$\backslash$};
            }
        },
        postaction={decorate}
    }
}

 \newcommand{\newuparrow}{{{\rlap{$\ $}\hbox{$\uparrow$}}}}
 \newcommand{\twoheaddownarrow}{{\rlap{\rlap{$\ $}\raise .25ex\hbox{$\downarrow$}}\raise-.25ex\hbox{$\downarrow$}}}
 \newcommand{\twoheaduparrow}{{\rlap{\rlap{$\ $}\raise .25ex\hbox{$\uparrow$}}\raise-.25ex\hbox{$\uparrow$}}}

\newcommand{\set}[1]{\{\,#1\,\}}
\newcommand{\bigset}[1]{\bigl\{\,#1\,\bigr\}}
\newcommand{\tbigwedge}{\mathop{\textstyle \bigwedge }}
\newcommand{\tbigcap}{\mathop{\textstyle \bigcap }}
\newcommand{\tbigcup}{\mathop{\textstyle \bigcup }}

\newcommand{\tbigvee}{\mathop{\textstyle \bigvee }}
 
\makeatletter
\newcommand*{\@old@slash}{}\let\@old@slash\slash
\def\slash{\relax\ifmmode\delimiter"502F30E\mathopen{}\else\@old@slash\fi}
\makeatother

\def\R{\mathbb{R}}

\def\N{\mathbb{N}}

\def\SS{\mathsf{S}}

\def\cl{\mathfrak{c}}
\def\op{\mathfrak{o}}

\newtheorem{theorem}{Theorem}[section]
\newtheorem{proposition}[theorem]{Proposition}

\newtheorem{lemma}[theorem]{Lemma}
\newtheorem{corollary}[theorem]{Corollary}

\theoremstyle{definition}

\theoremstyle{remark}
\newtheorem{remark}[theorem]{Remark}

\title{Joins of closed sublocales are not always a coframe}
\author{Igor Arrieta}
\address{Department of Mathematics \\
University of the Basque Country UPV/EHU, 48080, Bilbao, Spain}
\email{igor.arrieta@ehu.eus}
\keywords{Locale, frame, coframe, joins of closed sublocales, weak subfitness}
\subjclass[2020]{18F70 (primary); 06D22, 54D10 (secondary)}
\begin{document}

\maketitle

\begin{abstract}
Given a locale $L$, the ordered collection $\SS_c(L)$ of joins of closed sublocales forms a frame—somewhat unexpectedly, as it is naturally embedded in the coframe of all sublocales of $L$, where by coframe we mean the order-theoretic dual of a frame. This construction has attracted attention in point-free topology: as a maximal essential extension in the category of frames, for its (non-)functorial properties, its relation to canonical extensions and exact filters of frames, etc.

A central open question of the theory, posed by Picado, Pultr, and Tozzi in 2019, asked whether $\SS_c(L)$ is always a coframe, or whether there exists a locale for which this fails. In this paper, we resolve this question in the negative by constructing a locale $L$ such that $\SS_c(L)$ is not a coframe. The main challenge in such question lies in the difficulty of understanding exact infima in $\SS_c(L)$; we circumvent this by analysing  a certain separation property satisfied by $\SS_c(L)$.
\end{abstract}

\section{Introduction}
This  note concerns the family $\SS_c(L)$ of joins of closed sublocales of a locale $L$, introduced by Picado, Pultr, and Tozzi in \cite{PicadoPultrTozzi2019}. They showed that under the weak point-free separation axiom of \emph{subfitness}, the set $\SS_c(L)$ forms a Boolean algebra that coincides with the Booleanization $\SS_b(L)$ of the coframe $\SS(L)$ of all sublocales. The latter was subsequently studied in detail in \cite{Arrieta2022}.

Beyond this setting, the structure of $\SS_c(L)$ has attracted increasing attention in point-free topology. It has been studied for its (non-)functorial properties \cite{Arrieta2022,ball2019}, its role as a maximal essential extension (or envelope) in the category of frames \cite{ball2018}, and as a discretization mechanism for modeling non-continuous real-valued maps on locales \cite{DISCB}. Further applications include point-free counterparts of classical topological properties (see e.g.\ \cite{TBB}), the analysis of certain classes of filters \cite{MoshierPultrSuarez2020}, and connections to canonical extensions of frames \cite{JaklSuarez2025}.

Somewhat surprisingly, one of the main results of \cite{PicadoPultrTozzi2019} is that $\SS_c(L)$ is \emph{always} a frame, even in the absence of subfitness. Moreover, it is naturally embedded as a join-sublattice of the coframe $\SS(L)$. This seemingly paradoxical situation—where a frame sits inside a coframe—has been explained in various ways: for instance, by viewing $\SS_c(L)$ as a subframe of $\SS_b(L)$ (cf.\ \cite{Arrieta2022}) or through general results about exactness of meets, as in \cite{MoshierPicadoPultr2022}; and also fits naturally in the theory of Raney extensions recently initiated by Suarez \cite{SuarezRaneyExtensions2024}.

However, as first pointed out in \cite{PicadoPultrTozzi2019}, it remained an {open problem} whether $\SS_c(L)$ is always a coframe, or whether there exists a locale $L$ for which $\SS_c(L)$ fails to be one. 
This open question has been echoed in subsequent literature—for example, in the work by Jakl, Moshier,  Picado, Pultr, Suarez and others (see e.g. \cite[Notes~5.1.1]{MoshierPultrSuarez2020}),  \cite[IX~6.4.1]{separation} or \cite[\S 10]{JaklSuarez2025}).

Indeed, for all known natural classes of locales, $\SS_c(L)$ satisfied the coframe property: we recall that if $L$ is subfit, then $\SS_c(L)$ is not only a coframe but indeed a Boolean algebra. If $L$ is a $T_D$-spatial locale, then $\SS_c(L)$ is again a coframe. Similarly, if $L$ is both a frame and a coframe, then so is $\SS_c(L)$. These cases suggest that the class of locales for which $\SS_c(L)$ fails to be a coframe is rather restricted.

In this paper, we provide a solution to this problem in the negative: we construct an explicit locale $L$ such that $\SS_c(L)$ is not a coframe. The main challenge in addressing such questions lies in the difficulty of identifying exact meets within $\SS_c(L)$. Rather than tackling this directly, we circumvent it by analysing when $\SS_c(L)$ is weakly subfit, which allows us to detect the failure of coframeness by using an earlier construction by Paseka and \v Smarda.

The organization of the paper is as follows. In Section~\ref{sec2} we provide necessary notation and background on the categories of locales and frames, whereas in Section~\ref{sec3} we state the main result. In Subsection~\ref{subse3} we study the resulting separation axiom in a broader context, which may be of interest in its own right, and outline possible directions for future work.

\section{Preliminaries}\label{sec2}
Our notation and terminology regarding the categories of frames and locales will be that of \cite{PP12} (see also \cite{STONE}). The Heyting operator in a frame $L$, right adjoint to the meet operator, will be denoted by $\to$; for each $a\in L$, $a^*=a\to 0$ is the pseudocomplement of $a$.  An element $a\in L$ is \emph{dense} if $a^{*}=0$. Moreover, the properties $a\leq a^{**}$ and $a^{*}=a^{***}$ hold for any $a\in L$.
 A \emph{sublocale} of a locale $L$ is a subset $S\subseteq L$ closed under arbitrary meets such that
\[\forall a\in L\ \ \forall s\in S \ \ (a\to s\in S).\]
Sublocales of $L$ correspond to the regular subobjects of $L$ in the category of locales.
The system $\SS  (L)$ of all sublocales of $L$, partially ordered by inclusion, is a coframe \cite[Theorem~III.3.2.1]{PP12}, that is, its dual lattice is a frame.  Infima and suprema are given by
\[
\tbigwedge_{i\in I}S_i=\tbigcap_{i\in I}S_i, \quad \tbigvee_{i\in I}S_i=\{\tbigwedge M\mid M\subseteq\tbigcup_{i\in I} S_i\}.
\]
The least element is the sublocale $\mathsf{O}=\{1\}$ and the greatest element is the entire locale $L$.  We say that a sublocale $S$ of $L$ is \emph{proper} if $S\ne L$. Since $\SS(L)$ is a coframe, every sublocale $S$ of $L$ has a \emph{supplement} denoted by $S^{\#}$. For any $a\in L$, the sublocales
\[
\mathfrak{c}_L(a)=\newuparrow  a=\{x\in L\mid x\ge a\}\ \text{ and }\ \mathfrak{o}_L(a)=\{a\to b\mid b\in L\}
\]
are the \emph{closed} and \emph{open} sublocales of $L$, respectively (that we shall denote simply by $\mathfrak{c}(a)$ and $\mathfrak{o}(a)$ when there is no danger of confusion). For each $a\in L$, $\mathfrak{c}(a)$ and $\mathfrak{o}(a)$ are
complements of each other in $\SS(L)$
and satisfy the identities
\begin{equation*}{\label{identities.basic}}
\tbigcap_{i\in I} \mathfrak{c}(a_i)=\cl(\tbigvee_{i\in I} a_i),\quad \cl(a)\vee\cl(b)=\cl(a\wedge b),
\end{equation*}
\[\tbigvee_{i\in I}\op(a_i)=\op(\tbigvee_{i\in I} a_i) \quad\mbox{ and }\quad \op(a)\cap \op(b)=\op(a\wedge b).
\]

 The point-free separation axiom of subfitness will play an important role throughout the paper. A locale $L$ is said to be \emph{subfit} if
$$a\not\leq b\implies \exists c\in L \textrm{ such that } a\vee c=1  \ne b\vee c.$$
For spaces, subfitness is slightly weaker than $T_1$. Moreover, a locale is subfit if and only if each open sublocale is a join of closed sublocales.

\subsection{Booleanization} Given a locale $L$, we denote by  $B_L$ the subset consisting of \emph{regular elements} of $L$; that is, those $a\in L$ with $a^{**}=a$, or equivalently those $a\in L$ with $a=b^*$ for some $b\in L$. In other words, one has  $$B_L=\set{a^*\mid a\in L} =\set{a\in L\mid a^{**}=a};$$
 and this subset is called the \emph{Booleanization} of $L$. This sublocale can be characterized in several ways, e.g. it is the least dense sublocale, or equivalently, the unique Boolean dense sublocale. We recall that the join of a family $\{a_i\}_{i\in I}\subseteq B_L$ in $B_L$ is given by $(\bigvee_{i\in I}a_i)^{**}$.

\subsection{Joins of closed sublocales} \label{JCS} Let $\SS_c(L)$ denote the subset of $\SS(L)$ consisting of joins of closed sublocales i.e.
$$\SS_c(L)=\bigset{ \tbigvee_{a\in A} \cl(a)\mid A \subseteq L},$$
 endowed with the inclusion order inherited from $\SS(L)$. In  \cite{PicadoPultrTozzi2019}, Picado, Pultr and Tozzi  show (a.o.) that $\SS_c(L)$ is \emph{always} a frame which is embedded as a join-sublattice in the coframe $\SS(L)$. One of the main results from \cite{PicadoPultrTozzi2019} is that 
\begin{quote}\emph{if $L$ is subfit, and only in that case, $\SS_c(L)^{op}$ is a Boolean algebra and coincides precisely with the Booleanization of $\SS(L)^{op}$}.
\end{quote}

\section{The counterexample}\label{sec3}

\subsection{A construction of a non-spatial frame}  The following construction goes back to \cite[Example~2]{MurchistonStanley1984}, which was later used in \cite[Proposition~2.3]{PasekaSmarda1992} to produce a non-subfit Hausdorff frame (see also \cite[III~3.5]{separation}).  Let $L$ be a frame and set 
$$K(L):= \{(a,b)\in L\times B_L \mid a\leq b\}.$$
Clearly, $K(L)$ is a subframe of $L\times B_L$ (with the componentwise frame structure), and it is also closed under arbitrary meets. 
\begin{lemma}\label{mainlemma}
Let $U$ be a regular open in $\Omega(\R)$, the usual topology on the real line. Then
$$U\cup\{0\} = \tbigcap_{n\in \N} \textup{int} \left( \overline{U}\cup [-\tfrac{1}{n},\tfrac{1}{n}] \right) = \tbigcap_{n\in \N} \left(U\cup (-\tfrac{1}{n},\tfrac{1}{n})\right)^{**}.$$
Moreover, $$U=\textup{int}\left(\tbigcap_{n\in \N} \left(U\cup (-\tfrac{1}{n},\tfrac{1}{n})\right)^{**}\right)=\tbigwedge_{n\in \N} \left(U\cup (-\tfrac{1}{n},\tfrac{1}{n})\right)^{**},$$
where the meet is taken in $\Omega(\R)$, or equivalently in $B_{\Omega(\R)}$.
\end{lemma}

\begin{proof}
The inclusion ``$\subseteq$'' is trivial so let us show the reverse one. Let  $x\not\in U$ such that $x\ne 0$. We want to show $x\not\in  \tbigcap_{n\in \N} \textup{int} \left( \overline{U}\cup [-\tfrac{1}{n},\tfrac{1}{n}] \right)$.  Since $x\ne 0$, pick an $N\in\N$ such that $\frac{1}{N}< |x|$. It then suffices to show that $x\not\in \textup{int} \left( \overline{U}\cup [-\tfrac{1}{N},\tfrac{1}{N}] \right)$.  Suppose otherwise that $x\in \textup{int} \left( \overline{U}\cup [-\tfrac{1}{N},\tfrac{1}{N}] \right)$. Then there is an $\epsilon>0$ such that \begin{equation}\label{bat}(x-\epsilon,x+\epsilon)\subseteq  \overline{U}\cup [-\tfrac{1}{N},\tfrac{1}{N}].\end{equation} Let $\epsilon'=\textup{min}\{\epsilon, |x|-\tfrac{1}{N}\}>0$. Note that then, if $t\in (x-\epsilon',x+\epsilon')$, 
one has $|x|-|t|\leq |x-t|<\epsilon'\leq |x|-\frac{1}{N}$
and so $|t|> \frac{1}{N}$, from which follows $$(x-\epsilon',x+\epsilon') \subseteq \R\smallsetminus[-\tfrac{1}{N},\tfrac{1}{N}]$$
Combining it with \eqref{bat}, we get
$$(x-\epsilon',x+\epsilon')\subseteq \overline{U}$$ because $\epsilon'\leq \epsilon$.
But, since $U$ is regular open, we have $x\not\in  \textup{int} (\overline{U})$ --- i.e. $x\in \overline{\R \smallsetminus \overline{U}}$. Hence $(x-\epsilon',x+\epsilon')\cap (\R\smallsetminus\overline{U})\neq \varnothing$. This leads to a contradiction.

Let us show the second part of the statement. By the previous part, it suffices to show that $\textup{int}(U\cup\{0\})\subseteq U$, and clearly we may assume $0\not\in U$.  Let $V$ be an open set such that $V\subseteq U\cup\{0\}$. We want to show that $V\subseteq U$ --- i.e. that $0\not\in V$. By way of contradiction, suppose that $0\in V$. Then there is an $\epsilon>0$ such that $(-\epsilon,\epsilon)\subseteq V\subseteq U\cup\{0\}$. Then $(-\epsilon,\epsilon)\smallsetminus\{0\}\subseteq U$ and taking closures $[-\epsilon,\epsilon]\subseteq \overline{U}$, and so $0\in (-\epsilon,\epsilon)\subseteq \textup{int}(\overline{U})=U$ because $U$ is regular open.
This leads to a contradiction.
\end{proof}

\begin{proposition}\label{mainprop}
In $K(\Omega(\R))$ the following properties hold:
\begin{enumerate}
\item\label{mainprop1} $(\R\smallsetminus\{0\},\R)\not\in \tbigvee_{n\in\N} \cl_{K(\Omega(\R))} \left( (-\tfrac{1}{n},\tfrac{1}{n}), (-\tfrac{1}{n},\tfrac{1}{n})  \right)$\textup;
\item\label{mainprop2} $\cl_{K(\Omega(\R))}\left((\R\smallsetminus\{0\},\R)\right) \vee  \tbigvee_{n\in\N} \cl_{K(\Omega(\R))} \left( (-\tfrac{1}{n},\tfrac{1}{n}), (-\tfrac{1}{n},\tfrac{1}{n})  \right) = K(\Omega(\R))$.
\end{enumerate}
\end{proposition}

\begin{proof}
(1) Suppose, by way of contradiction that $$(\R\smallsetminus\{0\},\R)\in \tbigvee_{n\in\N} \cl_{K(\Omega(\R))} \left( (-\tfrac{1}{n},\tfrac{1}{n}), (-\tfrac{1}{n},\tfrac{1}{n})  \right).$$ Then there is a family $\{(U_n,V_n)\}\subseteq K(\Omega(\R))$ with $((-\tfrac{1}{n},\tfrac{1}{n}) , (-\tfrac{1}{n},\tfrac{1}{n}))\subseteq (U_n,V_n)$ and such that $(\R\smallsetminus\{0\},\R) =\tbigwedge_{n\in \N} (U_n,V_n)$. Then for any $n\in\N$ one has $\R\smallsetminus\{0\}\subseteq U_n$ and since $(-\tfrac{1}{n},\tfrac{1}{n})\subseteq U_n$ it follows that $U_n=V_n=\R$, a contradiction.\\[2mm]
(2) Note that 
\begin{align*}\cl_{K(\Omega(\R))}\left((\R\smallsetminus\{0\},\R)\right) \vee  \tbigvee_{n\in\N} \cl_{K(\Omega(\R))} \left( (-\tfrac{1}{n},\tfrac{1}{n}), (-\tfrac{1}{n},\tfrac{1}{n})  \right) &\\=   \tbigvee_{n\in\N} \cl_{K(\Omega(\R))} \left( (-\tfrac{1}{n},\tfrac{1}{n})\smallsetminus\{0\}, (-\tfrac{1}{n},\tfrac{1}{n})\right),\end{align*}
 and let $(U,V)\in K(\Omega(\R))$ --- i.e. $U\in\Omega(\R)$, $V\in B_{\Omega(\R)}$ and $U\subseteq V$. Then clearly  
 \begin{equation}\label{xxe}U=\tbigcap_{n\in \N}  ((-\tfrac{1}{n},\tfrac{1}{n})\smallsetminus\{0\})\cup U = \tbigwedge_{n\in\N}  ((-\tfrac{1}{n},\tfrac{1}{n})\smallsetminus\{0\})\cup U.\end{equation}
Hence for each $n\in \N$, set $$U_n:= ((-\tfrac{1}{n},\tfrac{1}{n})\smallsetminus\{0\})\cup U\quad \text{and} \quad V_n:=(-\tfrac{1}{n},\tfrac{1}{n})\sqcup V$$ where $\sqcup$ denotes join in $B_{\Omega(\R)}$ --- i.e. join in $\Omega(\R)$ followed by double negation. Note that for each $n\in\N$ one has $U_n\subseteq (-\tfrac{1}{n},\tfrac{1}{n})\cup U\subseteq (-\tfrac{1}{n},\tfrac{1}{n})\cup V\subseteq \left((-\tfrac{1}{n},\tfrac{1}{n})\cup V\right)^{**}= V_n$. Hence $(U_n,V_n)\in K(\Omega(\R))$ and clearly $(U_n,V_n)\in \cl_{K(\Omega(\R))} \left( (-\tfrac{1}{n},\tfrac{1}{n})\smallsetminus\{0\}, (-\tfrac{1}{n},\tfrac{1}{n})\right)$ for each $n\in\N$.
By \eqref{xxe} and Lemma~\ref{mainlemma}, one has $\tbigwedge_{n\in \N} (U_n,V_n)=(U,V)$.
\end{proof}

Recall that a frame $L$ is \emph{weakly subfit} if for all $a\ne 0$ in $L$ there is a $c\ne 1$ such that $a\vee c= 1$ (see e.g. \cite{separation,PicadoPultr2016,Arrieta2024}  for more information on this $T_1$-type separation axiom for locales).

\begin{proposition}\label{weaksubScL}
The following are equivalent for a locale $L$:\begin{enumerate}
\item $\SS_c(L)$ is weakly subfit;
\item Every proper open sublocale of $L$ is contained in a proper join of closed sublocales;
\item Every proper open dense sublocale of $L$ is contained in a proper join of closed sublocales.
\end{enumerate}
\end{proposition}
\begin{proof}
Unwinding (1) yields: for all non-void $S\in \SS_c(L)$ (i.e. $S\ne \mathsf{O}$) there is a proper $T\in \SS_c(L)$ with $S^{\#}\subseteq T$. Now taking $S$ to be a closed sublocale of $L$ gives (2). Conversely, take $S=\tbigvee_i \cl(a_i)$ non-void. Then there is an $i_0$ with $\cl(a_{i_0})$ non-void, i.e. $\op(a_{i_0})$ proper. Hence there is a proper join of closed sublocales $T$ with $\op(a_{i_0})\subseteq T$ and so $S\vee T\supseteq \cl(a_{i_0})\vee T=L$.
Finally, the equivalence (2)$\iff$(3) follows because if $\op(a)$ is not dense, then we can take its own closure (which is proper) as the required join of closed sublocales.
\end{proof}

\begin{theorem}\label{thws}
The frame $\SS_c(K(\Omega(\R)))$ is weakly subfit. 
\end{theorem}

\begin{proof}
Let $(U,V)\in K(\Omega(\R))$ such that $(U,V)\ne (\R,\R)$. Then $U\ne \R$, so there is an $x_0\in \R$ with $U\subseteq \R\smallsetminus\{x_0\}$ and so $(U,V)\subseteq (\R\smallsetminus\{x_0\},\R)$. Hence  in view of Proposition~\ref{weaksubScL}\,(2) it suffices to show that the proper open sublocale $\op_{K(\Omega(\R))}\left((\R\smallsetminus\{x_0\}, \R)\right)$ is contained in a proper join of closed sublocales. Since $\R$ is a homogeneous space, it is clearly enough to show it for $x_0=0$. Let $S=\tbigvee_{n\in\N} \cl_{K(\Omega(\R))} \left( (-\tfrac{1}{n},\tfrac{1}{n}), (-\tfrac{1}{n},\tfrac{1}{n})  \right)$. By Proposition~\ref{mainprop}\,(\ref{mainprop1}), we see that $S$ is proper.  Moreover, by Proposition~\ref{mainprop}\,(\ref{mainprop2}) we see that $\op_{K(\Omega(\R))}\left((\R\smallsetminus\{0\},\R)\right)\subseteq S$, as required.
\end{proof}

The following result is well-known:

\begin{lemma}{(\cite[Proposition~6.1]{PicadoPultr2016})}\label{coframes}
In a weakly subfit frame $L$, the following formula holds for pseudocomplements
$$a^*=\tbigwedge \{ x\mid x\vee a=1\},\qquad a\in L.$$
Hence, if $L$ is additionally a coframe, then $a\vee a^*=1$ for any $a\in L$ and so $L$ is a Boolean algebra.
\end{lemma}

\begin{corollary}
The frame $\SS_c(K(\Omega(\R)))$ is not a coframe. 
\end{corollary}

\begin{proof}
If it were a coframe,  Lemma~\ref{coframes} and Theorem~\ref{thws} would imply the Booleanness  $\SS_c(K(\Omega(\R)))$. But this implies (see \cite[Theorem~3.5]{PicadoPultrTozzi2019}) that $K(\Omega(\R))$ is subfit, which is not (cf. \cite[Proposition~2.3]{PasekaSmarda1992}). 
\end{proof}

\subsection{What happens for spaces?}\label{subse3}
We will call a locale $L$ \emph{symmetric} if its satisfies any of the equivalent conditions of Proposition~\ref{weaksubScL}. This terminology will be justified in what follows. Recall that a topological space is said to be \emph{symmetric} if its  specialization preorder is symmetric (i.e. an equivalence relation). It is a weak separation axiom, which together with the $T_0$-axiom is equivalent to the $T_1$-property of spaces.

We end up with the observation that symmetric locales are a localic generalization of the well-known notion of symmetric space.  Given a topological space $X$, we denote by $U_c(X)$ the set of unions of closed subspaces of $X$. Clearly, $U_c(X)$ is closed under the formation of arbitrary joins in  $\mathcal{P}(X)$, and moreover it is also closed under  arbitrary meets because of the complete distributivity of $\mathcal{P}(X)$. Hence $U_c(X)$ is both a frame and a coframe, and it is anti-isomorphic to the lattice of saturated subsets of $X$ (recall that a subspace of $X$ is \emph{saturated} if it is an intersection of open subspaces of $X$).

\begin{proposition}
Let $X$ be a space. Then the following are equivalent:
\begin{enumerate}
\item $X$ is symmetric;
\item $U_c(X)$ is a Boolean algebra;
\item $U_c(X)$ is weakly subfit;
\item Every proper open subspace of $X$ is contained in a proper join of closed subspaces;
\item Every proper open dense subspace of $X$ is contained in a proper join of closed subspaces.
\end{enumerate}
\end{proposition}

\begin{proof} Conditions (3), (4) and (5) are equivalent by a simple argument similar to that of the proof of Proposition~\ref{weaksubScL}. Moreover, (2) and (3) are equivalent by virtue of Lemma~\ref{coframes} and the fact that $U_c(X)$ is a coframe.\\[2mm]
(1)$\implies$(4): Because any proper open set is contained in one of the form $X\smallsetminus\overline{\{x\}}$, it suffices to show condition (4) for proper open sets of the form $X\smallsetminus\overline{\{x\}}$. But clearly $X\smallsetminus\overline{\{x\}}\subseteq \tbigcup\left\{ \overline{\{y\}}\mid y\not\in\overline{\{x\}}\right\}$, and moreover $x$ is not contained in the right hand side because of the symmetry axiom.\\[2mm]
(4)$\implies$(1): Let $x\in X$. Then, there is a family $\{F_i\}_{i\in I}$ of closed subspaces such that $X\smallsetminus \overline{\{x\}}\subseteq \tbigcup_i F_i\subsetneq X$. It then follows that
$$X\smallsetminus \overline{\{x\}}\subseteq \tbigcup\left\{ \overline{\{y\}}\mid y\not\in\overline{\{x\}}\right\}\subseteq  \tbigcup_i F_i \subsetneq X.$$
By properness, there is a $z\in X$ such that for any  $w\in X$, $z\in \overline{\{w\}}$ implies  $w\in \overline{\{x\}}$. By taking $w=z$ we see that $z\in\overline{\{x\}}$. Let finally $y\in X$ such that $x\in\overline{\{y\}}$, then  $z\in \overline{\{y\}}$ by transitivity, and so by taking $w=y$ it follows  $y\in \overline{\{x\}}$.
\end{proof}

\begin{remark}
In a $T_D$-space $X$, there is a perfect correspondence between subspaces of $X$ and the corresponding induced sublocales of $\Omega(X)$ --- 	see \cite{PP12}. In this setting, $X$ is a symmetric space (and hence $T_1$) if and only if the locale $\Omega(X)$ is symmetric.
\end{remark}

To sum up, we began by studying locales $L$ for which $\SS_c(L)$—which can be seen as a kind of discretization of $L$—satisfies a weak $T_1$-type condition. We have now seen that this class of locales, which we termed \emph{symmetric locales}, turns out—somewhat unexpectedly—to correspond to a genuine $T_1$-type separation condition on $L$ itself, by analogy with the topological setting (recall here that the $T_1$-spaces are precisely the symmetric $T_0$ ones).

Whether localic symmetry is equivalent to other known localic separation conditions (as studied in \cite{separation})—and in particular to $T_1$-type properties such as those examined in \cite{Arrieta2024}—remains an open question and will be explored in future work.
\subsection*{Declarations}
\subsubsection*{Funding Declaration}
The author acknowledges support from the Basque Government (grant IT1483-22 and a postdoctoral fellowship of the Basque Government, grant POS-2022-1-0015).
\subsubsection*{Competing interests} The authors have no competing interests to declare that are relevant to the content of this article.


\begin{thebibliography}{99}
\bibitem{Arrieta2022}
I. Arrieta, 
{On joins of complemented sublocales}, 
\emph{Algebra Universalis} \textbf{83}, (2022), art. no. 1.

\bibitem{Arrieta2024}
I. Arrieta,
{Localic separation and the duality between closedness and fittedness},
\emph{Topology and its Applications} \textbf{342} (2024), art. no. 108785.

\bibitem{ball2019}
R. N. Ball, J. Picado, and A. Pultr,
Some aspects of (non) functoriality of natural discrete covers of locales,
 \emph{Quaestiones Mathematicae},
 \textbf{42} (2019), 701--715.

\bibitem{ball2018}
R. N. Ball and A.  Pultr,
Maximal essential extensions in the context of frames,
\emph{Algebra Universalis}
 \textbf{79}  (2018), art. no 32.
 

\bibitem{TBB}
T. Dube,
Maximal Lindel{\"o}f Locales, 
\emph{Applied Categorical Structures} 
\textbf{27} (2019), 687--702.


\bibitem{JaklSuarez2025}
T. Jakl and A. L. Suarez,
{Canonical extensions via fitted sublocales},
\emph{Applied Categorical Structures} \textbf{33} (2025), art. no. 10.

\bibitem{STONE}
P.\,T.~Johnstone,
\newblock \emph{Stone spaces}, 
 Cambridge Studies in Advanced Mathematics \textbf{3},
\newblock  Cambridge University Press, Cambridge, 1982.

\bibitem{MoshierPicadoPultr2022}
M. A. Moshier, J. Picado, and A. Pultr, 
\textit{Some general aspects of exactness and strong exactness of meets}, 
Topology and its Applications \textbf{309} (2022), art. no. 107906.

\bibitem{MoshierPultrSuarez2020}
M. A. Moshier, A. Pultr, and A. L. Suarez,
{Exact and Strongly Exact Filters},
\emph{Applied Categorical Structures} \textbf{28} (2020), 907--920.

\bibitem{MurchistonStanley1984}
G. S. Murchiston and M. G. Stanley,
{A ‘$T_1$’ space with no closed points, and a “$T_1$” locale which is not ‘$T_1$’},
\emph{Mathematical Proceedings of the Cambridge Philosophical Society}, \textbf{95} (1984), 421--422.

\bibitem{PasekaSmarda1992}
J. Paseka and B. \v Smarda,
$T_2$-frames and almost compact frames,
\emph{Czechoslovak Mathematical Journal} \textbf{42} (1992), 385--402.


%
%
%
%
%
%
%
%
%
%
%
%
%
%
%

%
%
%
%
\bibitem{PP12}
J.~Picado and A.~Pultr,
\newblock \emph{Frames and locales: Topology without points}, 
 Frontiers in Mathematics \textbf{28}
\newblock Springer, Basel, 2012.
%
%

\bibitem{PicadoPultr2016}
J.~Picado and A.~Pultr,
{New Aspects of Subfitness in Frames and Spaces},
\emph{Applied Categorical Structures} \textbf{24} (2016), 703–714.

\bibitem{DISCB}
J. Picado and A. Pultr
 A Boolean extension of a frame and a representation of discontinuity,
\emph{Quaestiones Mathematicae}
 \textbf{40} (2017), 1111--1125.

\bibitem{separation}
J.~Picado and A.~Pultr,
\newblock \emph{Separation in point-free topology}, 
\newblock Birkh\"auser/Springer, Cham, 2021.


\bibitem{PicadoPultrTozzi2019}
J. Picado, A. Pultr, and A. Tozzi,
{Joins of closed sublocales},
\emph{Houston Journal of Mathematics} \textbf{45} (2019), 21--38.

\bibitem{SuarezRaneyExtensions2024}
A.\,L. Suarez,
{Raney extensions: A pointfree theory of T0 spaces based on canonical extension}, 
\emph{Journal of Pure and Applied Algebra} \textbf{229} (2025), art. no. 108137.

%
\end{thebibliography}
\end{document}